\documentclass{amsart}
\numberwithin{equation}{section}

\newtheorem{theorem}{Theorem}[section]
\newtheorem{lemma}[theorem]{Lemma}

\newtheorem{proposition}[theorem]{Proposition}
\newtheorem{definition}[theorem]{Definition}

\begin{document}

\author{Ravshan Ashurov}
\address{Ashurov R: V. I. Romanovskiy Institute of Mathematics Uzbekistan Academy of Sciences, University street 9, Tashkent, 100174, Uzbekistan.}
\email{ashurovr@gmail.com}

\author{Elbek Husanov}
\address{Husanov E: V. I. Romanovskiy Institute of Mathematics Uzbekistan Academy of Sciences, University street 9, Tashkent, 100174, Uzbekistan.}
\email{elbekhusanov02@gmail.com}

\small

\title[Inverse problem of determining  ...] {Inverse problem of determining a time-dependent coefficient in the time-fractional subdiffusion equation}

\begin{abstract}

This paper explores the forward and inverse problems for a fractional subdiffusion equation characterized by time-dependent diffusion and reaction coefficients. Initially, the forward problem is examined, and its unique solvability is established. Subsequently, the inverse problem of identifying an unknown time-dependent reaction coefficient is addressed, with rigorous proofs of the existence and uniqueness of its solution. Both problems' existence and uniqueness are demonstrated using Banach's contraction mapping theorem. Notably, this work is the first to investigate direct and inverse problems for such equations with time-dependent coefficients.

{\it Key words}: Subdiffusion equation, Caputo fractional derivative, forward problem, inverse problem.
\end{abstract}

\maketitle

\section{INTRODUCTION}
\qquad The Gerasimov–Caputo fractional derivative of order $0 < \rho < 1$ for the function $h(t)$ is defined as (see, e.g., \cite{KilbasSrivstavaTrujillo} p. 92)
$$
D_t^{\rho} h(t) = \frac{1}{\Gamma(1 - \rho)} \int_0^t \frac{h'(\eta)}{(t - \eta)^\rho} \, d\eta, \quad t > 0,
$$
provided the right-hand side exists, where $\Gamma(\rho)$ is Euler's gamma function.    

Consider the time-fractional diffusion equation in the domain $(0,l)\times(0,T]$,
\begin{equation} \label{1.1}
    D_{t}^{\rho}u(x,t)-\sigma(t)u_{xx}(x,t)+q(t)u(x,t)=f(x,t),
\end{equation}
with the initial condition
\begin{equation} \label{1.2}
    u(x,0)=\phi(x),\qquad x \in [0,l],
\end{equation}
and boundary conditions
\begin{equation} \label{1.3}
u(0,t)=u(l,t)=0,\,\,\,\,\,\phi(0)=\phi(l)=0,\qquad t \in [0,T],
\end{equation}
where $\sigma(t)$, $q(t)$, $f(x,t)$ are given smooth functions.

In this paper, if all the coefficients $\sigma(t)$, $q(t)$ and $f(x,t)$ are given, then problem \eqref{1.1}–\eqref{1.3} is called a forward problem. On the other hand, if the function $q(t)$ is unknown, then the problem of determining the pair of functions $\{u(x,t);q(t)\}$ is referred to as an inverse problem. To solve the inverse problem, we adopt the following additional condition:

\begin{equation} \label{1.4}
    {{\left. {{u}_{x}}(x,t) \right|}_{x=0}}= \psi(t), \qquad t\in [0,T],
\end{equation}
where $\psi(t)$ is given function.

The recovery of time- and space-dependent parameters in the diffusion equation has been studied by several authors (see, e.g., \cite{Kaban}). Inverse problems for parabolic equations are often investigated by comparison with inverse problems for hyperbolic equations. This approach is applicable to inverse problems with data on a portion of the boundary, including the quasi-stationary inverse problem for Maxwell’s equations \cite{Romanov}. The coefficient recovery problem is closely related to ill-posed problems in mathematical physics, and the main solution methods were developed by A. N. Tikhonov, M. M. Lavrentiev, A. I. Prilepko, V. G. Romanov, and S. I. Kabanikhin (see, e.g., \cite{Romanov}).

In problems similar to the one under consideration, both forward and inverse problems have been extensively studied for integer and fractional orders of derivatives under various additional conditions. Below, we briefly summarize some relevant results.

For the case $\rho = 1$ and $\sigma(t) = 1$, the existence and uniqueness of solutions to inverse problems concerning the determination of the coefficient $q$ have been thoroughly investigated in the works of N. Ya. Beznoshchenko \cite{Beznoshchenko1}–\cite{Beznoshchenko3}. Similarly, Kozhanov \cite{Kozhanov} also examined inverse problems for the determination of the coefficient $q$.

When $\sigma(t) = 1$ and $\rho \in (0,1)$, D. K. Durdiev et al. \cite{Durdiev1}–\cite{DurdievRahmonovBozorov} studied the determination of the coefficient $q$ under various additional conditions, analyzing the local existence of solutions. For $\rho \in (0,2)$ with $\rho \neq 1$, the problem of determining $q$ was considered by M. Yamamoto and L. Miller \cite{MillerYamamoto}, while for $\rho \in (1,2)$, this problem was addressed in \cite{TingWei}.

In the case where $\sigma(t)$ is time-dependent and $\rho \in (0,1)$, both forward problems and inverse problems for determining the coefficient $\sigma(t)$ have been investigated \cite{Zhang}–\cite{SerikbaevRuzhanskyTokmagambetov}. Additionally, \cite{SerikbaevRuzhanskyTokmagambetov2} analyzed the inverse problem related to the determination of the problem’s right-hand side.

A review of the existing literature indicates that, in the general case where $\sigma(t)$ and $q(t)$ are time-dependent functions, forward and inverse problems for equation \eqref{1.1} have not yet been explored. Therefore, this paper aims to fill this gap by analyzing the theoretical aspects of the problem for the case of time-dependent coefficients $\sigma(t)$ and $q(t)$. In particular, the existence, uniqueness, and principal properties of solutions to both forward and inverse problems are examined.

\section{PRELIMINARY MATERIALS} 

\subsection{Mittag-Leffler function and some of its important properties} In this subsection, we recall the definition of the generalized Mittag-Leffler function and present the essential properties that will be used in our analysis.

\begin{definition}\label{def2.1} (see, \cite{Dzhrbashyan}, p. 42). The two-parameter Mittag–Leffler function $E_{\rho ,\beta } (z)$ is defined as follows:
$$E_{\rho ,\beta } (z) = \sum _{k=0}^{\infty }\frac{z^{k}}{\Gamma \left(k\rho +\beta \right)}, \quad z \in {\mathbb C},$$
where $0 < \rho < 2$, $\beta \in \mathbb{C}$, and $\Gamma(\cdot)$ denotes the Gamma function.
\end{definition}

\begin{lemma}\label{lem2.2} (see \cite{Dzhrbashyan}, p. 134). If $0 < \rho < 2$, then the following estimate holds for any $t \ge 0$:
$$
|E_{\rho, \beta}(-t)| \le \frac{C}{1+t},
$$
where $C$ is a positive constant number, and $\beta$ is a real number.    
\end{lemma}

\begin{lemma}\label{lem2.3} (see \cite{SakamotoYamamoto}) Let $\lambda > 0$, $\rho > 0$ and a positive integer $m \in \mathbb{N}$. Then the following formula holds:
$$
\frac{d^m}{dt^m} E_{\rho,1}(-\lambda t^{\rho}) = -\lambda t^{\rho-m} E_{\rho,\rho-m+1}(-\lambda t^{\rho}), \quad t > 0.
$$
\end{lemma}

In particular, for $m = 1$, the formula becomes:
$$
\frac{d}{dt} E_{\rho,1}(-\lambda t^{\rho}) = -\lambda t^{\rho - 1} E_{\rho,\rho}(-\lambda t^{\rho}).
$$

\begin{proposition}\label{prop2.4} (see \cite{AshurovYusuf21})
Let $\lambda > 0$ and $\rho > 0$. Then the following integral identity holds:
$$
\int\limits_{0}^{t} \lambda \eta^{\rho-1} E_{\rho,\rho}(-\lambda \eta^{\rho}) \, d\eta = 1 - E_{\rho,1}(-\lambda t^{\rho}).
$$

\end{proposition}

\begin{lemma}\label{lem2.5} (see \cite{SerikbaevRuzhanskyTokmagambetov}) 
If $\rho \in (0,1]$, $\lambda>0$ and $\beta \ge \rho$, then the Mittag-Leffler function $E_{\rho ,\beta}(-\lambda t^{\rho})$ satisfies the following inequality:
$$
0 \le E_{\rho ,\beta}(-\lambda t^{\rho}) \le \frac{1}{\Gamma(\beta)}, \quad t \ge 0.
$$

\end{lemma}

\subsection{Some auxiliary concepts and results} In this subsection, we recall several auxiliary notions and results which will be used in the sequel.

\begin{lemma} \label{lem2.9} (see, \cite{Zhang2}) Let $0<\rho<1$ and $v=v(t) \in C[0,T]$ with $D_t^{\rho}v(t)\in C[0,T]$. Then the following statements hold: \\
$(a)$ If $t_0 \in   (0,T]$ and $v({{t}_{0}})=\underset{t\in [0,T]}{\mathop{\max }}\,v(t)$, then we have $D_t^{\rho}v(t) \ge 0$;\\
$(b)$ If $t_0 \in   (0,T]$ and $v({{t}_{0}})=\underset{t\in [0,T]}{\mathop{\min }}\,v(t)$, then we have $D_t^{\rho}v(t) \le 0$.    
\end{lemma}

\begin{theorem} \label{thm2.10} \textbf{Banach fixed point theorem.} (see, \cite{KilbasSrivstavaTrujillo}, 84 p.) Let $X$ be a Banach space and let $A : X \to X$ be a map such that  
$$
\|Au - Av\| \leq \beta \|u - v\|, \quad (0 < \beta < 1),
$$
holds for all $u, v \in X$. Then the operator $A$ has a unique fixed point $u^* \in X$, that is, $Au^* = u^*$.
\end{theorem}

\section{INVESTIGATION OF FORWARD PROBLEM}

\qquad In this section, we consider \eqref{1.1}–\eqref{1.3}, establish the existence and uniqueness of the solution, and derive certain regularity results. Throughout this section, we assume that the functions $\sigma(t)$, $q(t)$, $\phi(x)$ and $f(x,t)$ satisfy the following conditions:\\
\textbf{Assumption 1.} We assume that: \\
$(1)\qquad \sigma(t)\in {{C}^{+}}\left[0,T\right]:=\left\{ \sigma (t)\in C\left[ 0,T \right]:\,0<m_{\sigma}\le\sigma(t) \le M_{\sigma},\,\,\,t\in \left[ 0,T \right] \right\}$; \\
$(2)\qquad q(t)\in\left\{  q(t) \in C[0,T] :\,-\frac{m_{\sigma}{{\pi }^{2}}}{{{l}^{2}}}< n_q \le q(t) \le N_q< \frac{(M_{\sigma}-m_{\sigma}){{\pi }^{2}}}{{{l}^{2}} }\right\}$; \\
$(3) \qquad f(x,t)\in C([0,T];W_2^3(0,l)), \;\; with \;\; f(0,t)=f(l,t)={{\left. \frac{{{\partial }^{2}}f(x,t)}{\partial {{x}^{2}}} \right|}_{x=0}}={{\left. \frac{{{\partial }^{2}}f(x,t)}{\partial {{x}^{2}}} \right|}_{x=l}}=0$;\\
$(4) \qquad \phi(x)\in W_2^1(0,l), $\\
where $m_{\sigma}$, $M_{\sigma}$, $n_q$ and $N_q$ are constants.
 
\begin{definition}\label{def3.2} A function $u(x,t) \in {C([0,T] \times [0,l])}$ with the properties $D_t^{\rho}u(x,t)$, $u_{xx}(x,t) \in {C((0,T]\times [0,l])}$ and satisfying conditions \eqref{1.1}–\eqref{1.3} is called the classic solution of the forward problem.
\end{definition}

The main theorem of this section, related to problems \eqref{1.1}–\eqref{1.3}, is presented below.

\begin{theorem} \label{thm3.3} Let Assumption 1 be satisfied. Then the problem \eqref{1.1}–\eqref{1.3} has a unique solution $u(x,t)$.
\end{theorem}

We will seek the solution of problem \eqref{1.1}–\eqref{1.3} in the form:
\begin{equation}\label{3.1}
    u(x,t)=\sum\limits_{k=1}^{\infty }{{{u}_{k}}(t)\sin ({{\lambda }_{k}}x)},
\end{equation}
where ${{\lambda }_{k}}=\frac{\pi k}{l}$.

Taking equality \eqref{3.1} into account, we obtain the following problem:
\begin{equation}\label{3.2}
\left\{
\begin{aligned}
    &D_{t}^{\rho } u_{k}(t) + \lambda_k^2 \sigma(t) u_{k}(t)+q(t)u_{k}(t) = f_k(t), \qquad  0<t\le T; \\
    & u_{k}(0) = \phi_{k},
\end{aligned}
\right.
\end{equation}
where
$f_k(t)=\sqrt{\frac{2}{l}}\int\limits_{0}^{l}{f(x,t)\sin ({{\lambda }_{k}}x)dx}$ and $\phi_k=\sqrt{\frac{2}{l}}\int\limits_{0}^{l}{\phi(x)\sin ({{\lambda }_{k}}x)dx}$.

\begin{theorem} \label{thm3.4} Let Assumption 1 hold. Then, for each $k$, there exists a unique solution $u_k(t) \in C[0, T]$ to problem \eqref{3.2}, and this solution also satisfies $D_t^{\rho} u_k(t) \in C(0, T]$.
\end{theorem}

\begin{proof} 
Let us rewrite the Cauchy-type problem \eqref{3.2} in the form:
\begin{equation}\label{3.3}
\left\{
\begin{aligned}
    &D_{t}^{\rho }{{u}_{k}}(t)+\lambda _{k}^{2}{{M}_{\sigma }}{{u}_{k}}(t)=f_k(t)+\lambda _{k}^{2}({{M}_{\sigma }}-\sigma (t)){{u}_{k}}(t)-q(t){{u}_{k}}(t), \qquad 0<t\le T; \\
    & u_{k}(0) = \phi_{k}.
\end{aligned}
\right.
\end{equation}

The problem can be expressed in the form of the integral equation (see \cite{KilbasSrivstavaTrujillo}, p. 230):
$$
u_{k}(t) = \phi_{k} E_{\rho,1}\left(-\lambda^2_{k} M_{\sigma} t^{\rho} \right) 
$$
\begin{equation}\label{3.4}
+\int_{0}^{t}{\left[ f_k(s)+[\lambda _{k}^{2}({{M}_{\sigma }}-\sigma (s))-q(s)]{{u}_{k}}(s) \right]}{{(t-s)}^{\rho -1}}{{E}_{\rho ,\rho }}\left( -\lambda _{k}^{2}{{M}_{\sigma }}{{(t-s)}^{\rho }} \right)ds.
\end{equation}

Let us show that this integral equation has a unique solution in the space $C[0,T]$. To this end, we apply the Banach fixed point Theorem \ref{thm2.10}. Rewrite the integral equation \eqref{3.4} as follows:
$$
u_k(t) = A u_k(t),
$$
where  $Au_k(t)$ is defined as:
$$
Au_{k}(t) = \phi_{k} E_{\rho,1}\left(-\lambda^2_{k} M_{\sigma} t^{\rho} \right) 
$$
\begin{equation}\label{3.5}
+\int_{0}^{t}{\left[ f_k(s)+[\lambda _{k}^{2}({{M}_{\sigma }}-\sigma (s))-q(s)]{{u}_{k}}(s) \right]}{{(t-s)}^{\rho -1}}{{E}_{\rho ,\rho }}\left( -\lambda _{k}^{2}{{M}_{\sigma }}{{(t-s)}^{\rho }} \right)ds.
\end{equation}

To apply Theorem \ref{thm2.10}, we need to prove the following:

\qquad $(a)$ If $u_k (t) \in C[0,T]$, then $Au_{k}(t) \in C[0,T]$;

\qquad $(b)$ For any $v_k , w_k \in C[0,T]$, the following estimate holds:
\begin{equation}\label{3.6}
   \|Av_k - Aw_k\|_{C[0,T]} \leq \beta \|v_k - w_k\|_{C[0,T]}, \quad \beta \in (0,1). 
\end{equation}

First, note that if $u_k(t) \in C[0,T]$ then using conditions (1)–(3) of Assumption 1 and the properties of the Mittag–Leffler function, we conclude that $Au_k(t) \in C[0,T]$.

Second, let us prove the estimate \eqref{3.6}. Taking into account conditions (1)-(2) of Assumption 1 and using the operator defined in \eqref{3.5}, we obtain the following:
$$
\left| A{{v}_{k}}(t)-A{{w}_{k}}(t) \right|
$$
$$
\le {{\left\| {{v}_{k}}(t)-{{w}_{k}}(t) \right\|}_{C[0,T]}}\int_{0}^{t}{\left| \lambda _{k}^{2}({{M}_{\sigma }}-\sigma (s))-q(s) \right|{{(t-s)}^{\rho -1}}{{E}_{\rho ,\rho }}\left( -\lambda _{k}^{2}{{M}_{\sigma }}{{(t-s)}^{\rho }} \right)ds}
$$
$$
\le {{C}_{k}}{{\left\| {{v}_{k}}(t)-{{w}_{k}}(t) \right\|}_{C[0,T]}}\int_{0}^{t}{\lambda _{k}^{2}{{M}_{\sigma }}{{(t-s)}^{\rho -1}}{{E}_{\rho ,\rho }}\left( -\lambda _{k}^{2}{{M}_{\sigma }}{{(t-s)}^{\rho }} \right)ds},
$$
where ${{C}_{k}}=\underset{t\in [0,T]}{\mathop{\max }}\,\left| \frac{{{M}_{\sigma }}-{{m}_{\sigma }}}{{{M}_{\sigma }}}-\frac{q(t)}{\lambda _{k}^{2}{{M}_{\sigma }}} \right|$ and it is easy to see that $C_k \in (0,1)$ holds for all $\lambda_k$.

Based on Proposition \ref{prop2.4}, the following inequality can be established:
$$
\left| A{{v}_{k}}(t)-A{{w}_{k}}(t) \right|\le {{C}_{k}}{{\left\| {{v}_{k}}(t)-{{w}_{k}}(t) \right\|}_{C[0,T]}}.
$$

Consequently, there exists a unique solution $u_k \in C[0, T]$ of the problem \eqref{3.2}.

According to problem \eqref{3.2}, the function $u_k(t)$ satisfies the following equation:
$$
D_t^\rho u_k(t;\sigma; q) = f_k(t) - \lambda_k^2 \sigma(t) u_k(t) - q(t)u_k(t).
$$

Given conditions (1)–(3) of Assumption 1 and $u_k(t) \in C[0,T]$, we obtain that $D_t^\rho u_k(t) \in C[0,T]$.

\end{proof}

Let us move on to solving problem \eqref{1.1}–\eqref{1.3}. For this purpose, an estimate for $u_k(t)$ is presented first.

\begin{lemma} \label{lem3.5} Suppose that $u_k(t)$ is a unique solution of the problem \eqref{3.2}. Then following statements hold for each $k$:

\qquad $(a)$ If $f_k(t)\le0$ on $[0,T]$ and $\phi_k \le 0$, then we have $u_k(t) \le 0$ on $[0,T]$;

\qquad $(b)$ If $f_k(t)\ge0$ on $[0,T]$ and $\phi_k \ge 0$, then we have $u_k(t) \ge 0$ on $[0,T]$.
\end{lemma}

\begin{proof}  Since $u_k(t) \in C[0, T]$, there exists a point $t_0 \in [0, T]$ where the following equality holds:
$$
u_k(t_0) = \max_{t \in [0, T]} u_k(t).
$$

If $t_0 = 0$, then the inequality $u_k(t) \le u_k(0) = \phi_k \le 0$ holds. On the other hand, if $t_0 \in (0, T]$, applying Lemma \ref{lem2.9}, we deduce that $D_t^{\rho} u_k(t_0) \ge 0$. Therefore
$$
\lambda _{k}^{2}\sigma (t_0){{u}_{k}}(t_0)+q(t_0){{u}_{k}}(t_0)=-D_{t}^{\rho }{{u}_{k}}(t_0)+{{f}_{k}}(t_0)\le 0,
$$
and it follows that $u_k(t)\le 0$ on the interval $[0,T]$.

For case (b), let us consider $\tilde{u}_{k}(t) = -u_{k}(t)$. Then, by repeating the above proof for $\tilde{u}_{k}(t)$, we obtain that $\tilde{u}_{k}(t) \le 0$, which implies that $u_{k}(t) \ge 0$. 

\end{proof}

Let us decompose the Cauchy problem \eqref{3.2} into the homogeneous and inhomogeneous parts:

\begin{equation}\label{3.7}
\left\{
\begin{aligned}
    &D_{t}^{\rho }{{V}_{k}}(t)+\lambda _{k}^{2}\sigma(t){{V}_{k}}(t)+q(t){{V}_{k}}(t)=f_k(t);\\
    & V_{k}(0) = 0,
\end{aligned}
\right.
\end{equation}
and
\begin{equation}\label{3.8}
\left\{
\begin{aligned}
    &D_{t}^{\rho }{{W}_{k}}(t)+\lambda _{k}^{2}\sigma(t){{W}_{k}}(t)+q(t){{W}_{k}}(t)=0;\\
    & W_{k}(0) = \phi_{k}.
\end{aligned}
\right.
\end{equation}
Then $u_k(t) = V_k(t) + W_k(t)$.

To analyze the regularity properties of the solution, we represent the function $f_k(t)$ as follows:
$$
f_k(t) = f_k^{+}(t) + f_k^{-}(t), \qquad on \,\, [0,T],
$$
where $f_k^{+}(t)=\max\{0, f_k(t)\}$ and $f_k^{-}(t)=\min\{0, f_k(t)\}$.

\begin{lemma}\label{lem3.6} Assume that conditions (1)–(3) of Assumption 1 are satisfied. Then, for each function $V_k(t)$, the following estimate holds for all $t \in [0, T]$:
\begin{equation}\label{3.9}
     \left| V_{k}(t) \right| \le \int_{0}^{t} \left| f_k(s) \right |{{(t-s)}^{\rho -1}}{{E}_{\rho ,\rho }}\left( -(\lambda _{k}^{2}{{m}_{\sigma }} +n_q){{(t-s)}^{\rho }} \right)ds.
\end{equation}
\end{lemma}

\begin{proof}   By substituting $\sigma(t) = m_{\sigma}$ and $q(t) = n_q$ into equation \eqref{3.7}, we consider the problem:
$$
   \left\{
\begin{aligned}
    &D_{t}^{\rho }{\tilde{V}_{k}}(t)+\lambda _{k}^{2}m_{\sigma}{\tilde{V}_{k}}(t)+n_q{\tilde{V}_{k}}(t)=f_k(t);\\
    & \tilde{V}_{k}(0)= 0.
\end{aligned}
\right. 
$$

For the difference between the functions $\tilde{V}_{k}(t)$ and $V_{k}(t)$, we obtain the integral equation:
$$
\tilde{V}_{k}(t)-V_{k}(t)
$$
\begin{equation}\label{3.10}
    = \int_{0}^{t}{\left[ \lambda _{k}^{2}(\sigma(s)-m_{\sigma})+q(s)-n_q \right]}V_{k}(s){{(t-s)}^{\rho -1}}{{E}_{\rho ,\rho }}\left( -(\lambda _{k}^{2}{{m}_{\sigma }}+n_q){{(t-s)}^{\rho }} \right)ds. 
\end{equation}

Using \eqref{3.10} and Lemma \ref{lem3.5}, we obtain the estimate:
\begin{equation}\label{3.11}
\left\{
\begin{aligned}
    &0 \le V^{f^{+}}_{k}(t) \le \tilde{V}^{f^{+}}_{k}(t), \qquad on \,\, [0,T]; \\
    & \tilde{V}^{f^{-}}_{k}(t) \le V^{f^{-}}_{k}(t) \le 0, \qquad on \,\, [0,T],
\end{aligned}
\right.
\end{equation}
where the functions $V^{f^{+}}_{k}(t)$, $\tilde{V}^{f^{+}}_{k}(t)$ and $V^{f^{-}}_{k}(t)$, $\tilde{V}^{f^{-}}_{k}(t)$ denote, respectively, the solutions of the equations associated with the functions $f_k^{+}(t)$ and $f_k^{-}(t)$.

The solutions $\tilde{V}^{f^{+}}_{k}(t)$ and $\tilde{V}^{f^{-}}_{k}(t)$ are given, respectively, by the following expressions (see \cite{KilbasSrivstavaTrujillo}, p. 230):
$$
   \tilde{V}^{f^{+}}_{k}(t) = \int_{0}^{t}f^{+}_k(s){{(t-s)}^{\rho -1}}{{E}_{\rho ,\rho }}\left( -(\lambda _{k}^{2}{{m}_{\sigma }} + n_q){{(t-s)}^{\rho }} \right)ds, \qquad \forall t \in [0, T],
$$
and
$$
   \tilde{V}^{f^{-}}_{k}(t) = \int_{0}^{t}f^{-}_k(s){{(t-s)}^{\rho -1}}{{E}_{\rho ,\rho }}\left( -(\lambda _{k}^{2}{{m}_{\sigma }}+n_q){{(t-s)}^{\rho }} \right)ds, \qquad \forall t \in [0, T].
$$

Using $V^{f}_{k}(t)=V^{f^{+}}_{k}(t)+V^{f^{-}}_{k}(t)$ and \eqref{3.11}, we obtain:
$$
\left|V_{k}(t)\right| \le \int_{0}^{t} \left| f_k(s) \right |{{(t-s)}^{\rho -1}}{{E}_{\rho ,\rho }}\left( -(\lambda _{k}^{2}{{m}_{\sigma }} +n_q){{(t-s)}^{\rho }} \right)ds, \qquad \forall t \in [0, T].
$$

\end{proof}

\begin{lemma}\label{lem3.7}Assume that conditions (1)–(2) and (4) of Assumption 1 are satisfied. Then, for each function $W_k(t)$, the following estimate is valid for all $t \in [0,T]$:
\begin{equation}\label{3.12}
    \left| W_{k}(t) \right| \le \left| \phi_k \right| E_{\rho,1}\left(-(\lambda^2_{k} m_{\sigma}+ n_q) t^{\rho} \right).
\end{equation}
\end{lemma}
Lemma \ref{lem3.7} is proved similarly to Lemma \ref{lem3.6}.

By applying Lemmas \ref{lem3.6} and \ref{lem3.7}, the following estimate holds for the function $u_k(t)$:
$$
\left | u_k(t) \right | \le \left| \phi_k \right| E_{\rho,1}\left(-(\lambda^2_{k} m_{\sigma}+ n_q) t^{\rho} \right) 
$$
$$
+  \int_{0}^{t} \left| f_k(s) \right |{{(t-s)}^{\rho -1}}{{E}_{\rho ,\rho }}\left( -(\lambda _{k}^{2}{{m}_{\sigma }} +n_q){{(t-s)}^{\rho }} \right)ds.
$$

According to Bessel’s inequality, the following statement holds.

\begin{lemma}\label{lem3.8} If conditions (3) and (4) of Assumption 2 are satisfied, then the estimates hold true:
\begin{equation} \label{3.13}
\sum\limits_{k=1}^{\infty }{\lambda _{k}^{2}\left| {{f}_{k}}(t) \right|}\le \frac{l}{\sqrt{6}}{{\left\| f_{xxx}^{(3)}(x,t) \right\|}_{C([0,T];{{L}_{2}}(0,l))}}
\end{equation}
and
\begin{equation}\label{3.14}
   \sum\limits_{k=1}^{\infty }{\left| {{\phi }_{k}} \right|}\le \frac{l}{\sqrt{6}}{{\left\| \phi'(x) \right\|}_{{{L}_{2}}(0,l)}}.
\end{equation}
\end{lemma}

Next, for the function $u_{xx}(x,t)$, we have:
$$
\left| {{u}_{xx}}(x,t) \right|\le \sum\limits_{k=1}^{\infty }{\lambda _{k}^{2}\left| {{u}_{k}}(t) \right|}\le \sum\limits_{k=1}^{\infty }{\lambda _{k}^{2}\left| {{V}_{k}}(t) \right|}+\sum\limits_{k=1}^{\infty }{\lambda _{k}^{2}\left| {{W}_{k}}(t) \right|}={{Q}_{1}}+{{Q}_{2}},
$$
where ${{Q}_{1}}=\sum\limits_{k=1}^{\infty }{\lambda _{k}^{2}\left| {{V}_{k}}(t) \right|}$ and ${{Q}_{2}}=\sum\limits_{k=1}^{\infty }{\lambda _{k}^{2}\left| {{W}_{k}}(t) \right|}$.

Let us estimate each sum separately. The sum $Q_1$ is estimated using \eqref{3.9}:
$$
{{Q}_{1}}\le \sum\limits_{k=1}^{\infty }{\lambda _{k}^{2}\int_{0}^{t}{\left| {{f}_{k}}(s) \right|}{{(t-s)}^{\rho -1}}{{E}_{\rho ,\rho }}\left( -({{\lambda }^2_{k}}{{m}_{\sigma }}+{{n}_{q}}){{(t-s)}^{\rho }} \right)ds}.
$$

By employing \eqref{3.13} and Lemma \ref{lem2.5}, we obtain:
\begin{equation}\label{3.15}
{{Q}_{1}}\le \frac{l{{T}^{\rho }}}{\sqrt{6}\Gamma (\rho +1)}{{\left\| f_{xxx}^{(3)}(x,t) \right\|}_{C([0,T];{{L}_{2}}(0,l))}}. 
\end{equation}

To estimate the sum $Q_2$, we use \eqref{3.12},\eqref{3.14} and Lemma \ref{lem2.2}:
\begin{equation}\label{3.16}
{{Q}_{2}}=\sum\limits_{k=1}^{\infty }{\lambda _{k}^{2}\left| {{\phi }_{k}} \right|{{E}_{\rho ,1}}\left( -(\lambda _{k}^{2}{{m}_{\sigma }}+{{n}_{q}}){{t}^{\rho }} \right)}
\le \frac{Cl}{\sqrt{6}}t^{-\rho}{{\left\| \phi'(x) \right\|}_{{{L}_{2}}(0,l)}}, \qquad t>0.   
\end{equation}

Using \eqref{3.15} and \eqref{3.16}, the following estimate is obtained:
$$
\left| {{u}_{xx}}(x,t) \right|\le \frac{l{{T}^{\rho }}}{\sqrt{6}\Gamma (\rho +1)}{{\left\| f_{xxx}^{(3)}(x,t) \right\|}_{C([0,T];{{L}_{2}}(0,l))}}+\frac{Cl}{\sqrt{6}}{{t}^{-\rho }}{{\left\| {\phi }'(x) \right\|}_{{{L}_{2}}(0,l)}}, \qquad t>0.
$$

Hence, it follows that $u_{xx}(x,t) \in C((0,T]\times [0,l])$. Using the same reasoning as above, $u(x,t) \in C([0,T]\times [0,l])$ is established.

We estimate the function $D_t^{\rho} u(x,t)$ from equation \eqref{1.1} using the estimates of the functions $u(x,t)$ and $u_{xx}(x,t)$ as follows:
$$
{{\left| D_{t}^{\rho }u(x,t) \right|}}\le \frac{l}{\sqrt{6}}\left( \frac{({{M}_{\sigma }}+{{N}_{q}}){{T}^{\rho }}}{\Gamma (\rho +1)}+1 \right){{\left\| f_{xxx}^{(3)}(x,t) \right\|}_{C([0,T];{{L}_{2}}(0,l))}}
$$
$$
+\frac{Cl({{M}_{\sigma }}+{{N}_{q}})}{\sqrt{6}}{{t}^{-\rho }}{{\left\| \phi' (x) \right\|}_{{{L}_{2}}(0,l)}}, \qquad\qquad t>0.
$$

Hence, it follows that $D_{t}^{\rho }u(x,t) \in C((0,T]\times [0,l])$.

Let us investigate the uniqueness of the solution to problem \eqref{1.1}-\eqref{1.3}. Suppose we have two solutions: $\tilde{u}(x,t)$, $\bar{u}(x,t)$ and set $u(x,t) = \tilde{u}(x,t) - \bar{u}(x,t)$. Then, we have 
\begin{equation} \label{3.17}
  \left\{
\begin{aligned}
    &D_{t}^{\rho}u(x,t)-\sigma(t)u_{xx}(x,t)+q(t)u(x,t)=0, \qquad    (x,t) \in (0,l) \times (0,T]; \\
    &  u(x,0)=0,\qquad x \in [0,l].
\end{aligned}
\right.  
\end{equation}

Let $u_k(t) = \sqrt{\frac{2}{l}} (u(x,t), sin(\lambda_kx))$. Then, according to Definition \ref{def3.2} of the solution:
$$
D_t^\rho u_k(t) =  \sqrt{\frac{2}{l}}(D_{t}^{\rho}u(x,t), sin(\lambda_kx)) = \sqrt{\frac{2}{l}}\sigma(t)(u_{xx}(x,t), sin(\lambda_kx))-\sqrt{\frac{2}{l}}q(t)(u(x,t), sin(\lambda_kx))
$$
$$
=-\lambda_k^2 \sigma(t) u_k(t) - q(t)u_k(t).
$$

From the above equality and the second condition of \eqref{3.17}, we come to the following Cauchy problem:
$$
\left\{
\begin{aligned}
    &D_{t}^{\rho } u_{k}(t) + \lambda_k^2 \sigma(t) u_{k}(t)+q(t)u_{k}(t) = 0, \qquad  0<t\le T; \\
    & u_{k}(0) = 0.
\end{aligned}
\right.
$$

Using estimates \eqref{3.9} and \eqref{3.12}, we deduce that $u_k(t)=0$ for all $k$. Therefore, since the system $\{\sin(\lambda_k x)\}$ is complete in $L^2(0,l)$, it follows that $u(x,t)\equiv 0$.

\hfill $\square$

\section{INVESTIGATION OF THE INVERSE PROBLEM}

In this section, we study the inverse problem \eqref{1.1}–\eqref{1.4}, which concerns the determination of the pair of functions $\{u(x,t); q(t)\}$.

Following the work of M. Ruzhansky and his colleagues \cite{SerikbaevRuzhanskyTokmagambetov}, we introduce the following conditions on the functions $f(x,t)$ and $\phi(x)$.\\
\textbf{Assumption 2.} We assume that:\\
$(1) \qquad f(x,t)\in C([0,T];W_{2}^{4}(0,l))\;with\;{{f}_{k}}(t)\ge 0\;on\;[0,T]\;for\;all\;k,\;and\;f(0,t)=f(l,t)={{\left. \frac{{{\partial }^{2}}f(x,t)}{\partial {{x}^{2}}} \right|}_{x=0}}={{\left. \frac{{{\partial }^{2}}f(x,t)}{\partial {{x}^{2}}} \right|}_{x=l}}=0$; \\
$(2) \qquad \phi(x)\in W_2^4(0,l)  \,\, with \,\,\phi_k \ge 0  \,\,for \,\, all\,\,k, \,\, and \,\,\,{{\left. \frac{d^2\phi(x)}{d x^2} \right|}_{x=0}}={{\left. \frac{d^2\phi(x)}{d x^2} \right|}_{x=l}}=0.$ 

\begin{definition}\label{def4.2} A pair of functions $\{u(x,t),q(t)\}$ with the properties  $u(x,t) \in C_x^1([0,T] \times [0,l])$, $D_t^{\rho}u(x,t)$, $u_{xx}(x,t) \in C_x^1((0,T] \times [0,l])$, $q(t) \in C[0,T]$ and satisfying conditions \eqref{1.1}-\eqref{1.4} is called the solution of the inverse problem.

\end{definition}
    
We begin by presenting the following auxiliary lemma, which is needed before solving the inverse problem \eqref{1.1}–\eqref{1.4}.

According to Parseval’s equality, the following statement holds.

\begin{lemma}\label{lem4.4} Let Assumption 2 hold. Then the following estimates hold true:
\begin{equation} \label{4.1}
    \sum\limits_{k=1}^{\infty }{\lambda _{k}^{3}\left| {{f}_{k}}(t) \right|}\le \frac{l}{\sqrt{6}}{{\left\| f_{xxxx}^{(4)}(x,t) \right\|}_{{{L}_{2}}(0,l)}}
\end{equation}
and
\begin{equation}\label{4.2}
   \sum\limits_{k=1}^{\infty }{\lambda _{k}^{3}\left| {{\phi }_{k}} \right|}\le \frac{l}{\sqrt{6}}{{\left\| \phi^{(4)}(x) \right\|}_{{{L}_{2}}(0,l)}}.
\end{equation}
\end{lemma}

\begin{theorem}\label{thm4.3} Let Assumption 2 be satisfied. Then the inverse problem \eqref{1.1}–\eqref{1.4} admits a unique solution $\{u(x,t);q(t)\}$, provided that the following conditions hold:\\
$(1)\qquad \psi (t) \in C^1[0,T] \,\, and \,\,  \psi (t) \ge {{\psi }_{0}}>0$; \\
$(2) \qquad  {{\left. {{\phi }_{x}}(x) \right|}_{x=0}} = {{\left. \psi (t) \right|}_{t=0}}$; \\
$(3)\qquad 0 \le \frac{{{T}^{\rho }}}{\psi (t)}\left( {{f}_{x}}(0,t)-D_{t}^{\rho }\psi (t) \right)<\frac{{{T}^{\rho }}{{\pi }^{2}}}{{{l}^{2}}}({{M}_{\sigma }}-{{m}_{\sigma }})-\Gamma (\rho +1) $; \\
$(4) \qquad \frac{{{T}^{2\rho }}}{\Gamma (\rho +1)}{{\left\| f_{xxxx}^{(4)}(x,t) \right\|}_{C([0,T];{{L}_{2}}(0,l))}}+{{T}^{\rho }}{{\left\| \phi^{(4)}(x) \right\|}_{{{L}_{2}}(0,l)}}<\frac{{\sqrt{6}{\psi }_{0}}\Gamma (\rho +1)}{l{{M}_{\sigma }}}.$
\end{theorem}

To determine the function $q(t)$, we first differentiate equation \eqref{1.1} with respect to $x$. Taking into account the additional condition \eqref{1.4}, we obtain the following equality:

\begin{equation}\label{4.3}
    D_{t}^{\rho }\psi (t)-\sigma (t){{\left. {{u}_{xxx}}(x,t) \right|}_{x=0}}+q(t)\psi (t)={{\left. {{f}_{x}}(x,t) \right|}_{x=0}},
\end{equation}
where
$$
{{\left. {{u}_{xxx}}(x,t) \right|}_{x=0}}=-\sum\limits_{k=1}^{\infty }{\lambda _{k}^{3}{{u}_{k}}(t)}.
$$

It follows from equation \eqref{4.3} that the function $q(t)$ is given by:
\begin{equation}\label{4.4}
    q(t)=q_0(t)+\frac{\sigma (t)}{\psi (t)}\sum\limits_{k=1}^{\infty }{\lambda _{k}^{3}{{u}_{k}}(t)},
\end{equation}
where $q_0(t)=\frac{1}{\psi (t)}\left[ {{f}_{x}}(0,t)-D_{t}^{\rho }\psi (t) \right]$ and $u_k(t)$ denotes the solution of the integral equation \eqref{3.4}, which depends on the function $q(t)$.

\begin{lemma}\label{lem4.5} Suppose that Assumption 2 holds. Then the following two statements are equivalent:\\
$(1) \qquad    \sum\limits_{k=1}^{\infty }{\lambda _{k}^{3}\left| {{u}_{k}}(t) \right|}\le \frac{l{{T}^{\rho }}}{\sqrt{6}\Gamma (\rho +1)}{{\left\| f_{xxxx}^{(4)}(x,t) \right\|}_{C([0,T];L_2(0,l))}} + \frac{l}{\sqrt{6}}{{\left\| \phi^{(4)}(x) \right\|}_{{{L}_{2}}(0,l)}};$\\\\
$(2) \qquad Condition \; (2) \; of \; Assumption \;1:$

\qquad  \quad$ q(t)\in\left\{  q(t) \in C[0,T] :\,-\frac{m_{\sigma}{{\pi }^{2}}}{{{l}^{2}}}< n_q \le q(t) \le N_q< \frac{(M_{\sigma}-m_{\sigma}){{\pi }^{2}}}{{{l}^{2}} }\right\}.$
\end{lemma}

\begin{proof}
    
Using the decomposition $u_k(t) = V_k(t) + W_k(t)$, we obtain:
$$
\sum\limits_{k=1}^{\infty }{\lambda _{k}^{3}\left| {{u}_{k}}(t) \right|}\le \sum\limits_{k=1}^{\infty }{\lambda _{k}^{3}\left| {{V}_{k}}(t) \right|}+\sum\limits_{k=1}^{\infty }{\lambda _{k}^{3}\left| {{W}_{k}}(t) \right|}={{K}_{1}}+{{K}_{2}},
$$
where $ K_1=\sum\limits_{k=1}^{\infty }{\lambda _{k}^{3}\left| {{V}_{k}}(t) \right|}$ and $ K_2=\sum\limits_{k=1}^{\infty }{\lambda _{k}^{3}\left| {{W}_{k}}(t) \right|}$.

First, we estimate the first sum $K_1$ using \eqref{3.9} and \eqref{4.1}:
$$
K_1=\sum\limits_{k=1}^{\infty }{\lambda _{k}^{3}\left| {{V}_{k}}(t) \right|}\le \int_{0}^{t}{\sum\limits_{k=1}^{\infty }{\lambda _{k}^{3}\left| {{f}_{k}}(s) \right|}}{{(t-s)}^{\rho -1}}{{E}_{\rho ,\rho }}\left( -({\lambda^2_{k}}{{m}_{\sigma }+n_q)}{{(t-s)}^{\rho }} \right)ds
$$
$$
\le \frac{l}{\sqrt{6}}{{\left\| f_{xxxx}^{(4)}(x,t) \right\|}_{C([0,T];L_2(0,l))}} \int_{0}^{t}{{{(t-s)}^{\rho -1}}{{E}_{\rho ,\rho }}\left( -({\lambda^2_{k}}{{m}_{\sigma }+n_q)}{{(t-s)}^{\rho }} \right)ds}.
$$

By applying Lemma \ref{lem2.5}, we obtain:
\begin{equation}\label{4.5}
K_1 \le \frac{l{{T}^{\rho }}}{\sqrt{6}\Gamma (\rho +1)}{{\left\| f_{xxxx}^{(4)}(x,t) \right\|}_{C([0,T];L_2(0,l))}}.
\end{equation}

Second, we estimate the sum $K_2$ by using \eqref{3.12}:
$$
{{K}_{2}}=\sum\limits_{k=1}^{\infty }{\lambda _{k}^{3}\left| {{W}_{k}}(t) \right|}\le \sum\limits_{k=1}^{\infty }{\lambda _{k}^{3}\left| {{\phi }_{k}} \right|{{E}_{\rho ,1}}\left( -(\lambda _{k}^{2}{{m}_{\sigma }+n_q)}{{t}^{\rho }} \right)} \le \sum\limits_{k=1}^{\infty }{\lambda _{k}^{3}\left| {{\phi }_{k}} \right|}.
$$

By applying \eqref{4.2}, it follows that:

\begin{equation}\label{4.6}
    K_2 \le \frac{l}{\sqrt{6}}{{\left\| \phi^{(4)}(x) \right\|}_{{{L}_{2}}(0,l)}}.
\end{equation}

Based on the estimates given in \eqref{4.5} and \eqref{4.6}, we arrive at the following inequality:
$$
\sum\limits_{k=1}^{\infty }{\lambda _{k}^{3}\left| {{u}_{k}}(t) \right|}\le \frac{l{{T}^{\rho }}}{\sqrt{6}\Gamma (\rho +1)}{{\left\| f_{xxxx}^{(4)}(x,t) \right\|}_{C([0,T];L_2(0,l))}} + \frac{l}{\sqrt{6}}{{\left\| \phi^{(4)}(x) \right\|}_{{{L}_{2}}(0,l)}}.
$$

Next, condition (2) of Assumption 1 is examined. From Assumption 2 and conditions (1)–(2) of Theorem \ref{thm4.3}, it follows that $q(t) \ge 0$. Consequently, the following inequality is verified:
$$
{{q}_{0}}(t)+\frac{\sigma (t)}{\psi (t)}\sum\limits_{k=1}^{\infty }{\lambda _{k}^{3}{{u}_{k}}(t)}<\frac{({{M}_{\sigma }}-{{m}_{\sigma }}){{\pi }^{2}}}{{{l}^{2}}}.
$$

From condition (1) of Lemma \ref{lem4.5}, it follows that:
$$
{{q}_{0}}(t)+\frac{l{{M}_{\sigma }}{{T}^{\rho }}}{\sqrt{6}{{\psi }_{0}}\Gamma (\rho +1)}{{\left\| f_{xxxx}^{(4)}(x,t) \right\|}_{C([0,T];{{L}_{2}}(0,l))}}
$$
$$
+\frac{l{{M}_{\sigma }}}{\sqrt{6}{{\psi }_{0}}}{{\left\| {{\phi }^{(4)}}(x) \right\|}_{{{L}_{2}}(0,l)}}<\frac{({{M}_{\sigma }}-{{m}_{\sigma }}){{\pi }^{2}}}{{{l}^{2}}}.
$$

According to condition (4) of Theorem \ref{thm4.3}, the inequality takes the following form:
$$
{{q}_{0}}(t)+\frac{\Gamma (\rho +1)}{{{T}^{\rho }}}<\frac{({{M}_{\sigma }}-{{m}_{\sigma }}){{\pi }^{2}}}{{{l}^{2}}}.
$$

According to condition (3) of Theorem \ref{thm4.3}, the above inequality holds. Consequently, condition (2) of Assumption 1 is satisfied.

\end{proof} 

We introduce an operator $L$, defining it by the right-hand side of the equation \eqref{4.4}:
\begin{equation} \label{4.7}
     L[q](t)=q_0(t)+\frac{\sigma (t)}{\psi (t)}\sum\limits_{k=1}^{\infty }{\lambda _{k}^{3}{{u}_{k}}(t)}. 
\end{equation}

Then, the equation \eqref{4.4} is written in more convenient form as
$$
     q(t)=L[q](t).
$$

To apply Theorem \ref{thm2.10}, we need to prove the following:

\qquad $(a)$ If $q(t)\in C[0,T]$, then $L[q](t) \in C[0,T]$;

\qquad $(b)$ For any $\tilde{q}(t) , \bar{q}(t) \in C[0,T]$, the following estimate holds:
$$
   \| L\tilde{q}(t) - L\bar{q}(t)\|_{C[0,T]} \leq \beta \|\tilde{q}(t) - \bar{q}(t)\|_{C[0,T]}, \quad \beta \in (0,1). 
$$

If $q(t) \in C[0,T]$, then to show that $L[q](t) \in C[0,T]$, we apply condition (1) of Lemma \ref{lem4.5}:
$$
{{\left\| L[q](t) \right\|}_{C[0,T]}}\le {{\left\| {{q}_{0}}(t) \right\|}_{C[0,T]}}+\frac{{{M}_{\sigma }}}{{{\psi }_{0}}}\sum\limits_{k=1}^{\infty }{\lambda _{k}^{3}\left| {{u}_{k}}(t) \right|}
$$
$$
\le {{\left\| {{q}_{0}}(t) \right\|}_{C[0,T]}}+\frac{l{{M}_{\sigma }}{{T}^{\rho }}}{\sqrt{6}{{\psi }_{0}}\Gamma (\rho +1)}{{\left\| f_{xxxx}^{(4)}(x,t) \right\|}_{C([0,T];{{L}_{2}}(0,l))}}+\frac{l{{M}_{\sigma }}}{\sqrt{6}{{\psi }_{0}}}{{\left\| {{\phi }^{(4)}}(x) \right\|}_{{{L}_{2}}(0,l)}}.
$$

To prove part (b), we consider the following estimate for the two pairs of functions $\{ \tilde{u}_k (t);\tilde{q}(t) \}$ and $\{ \bar{u}_k (t);\bar{q}(t) \}$:
$$
    \left| L[\tilde{q}](t)-L[\bar{q}](t) \right|\le \frac{{{M}_{\sigma }}}{{{\psi }_{0}}}\sum\limits_{k=1}^{\infty }{\lambda _{k}^{3}\left| \tilde{u}_k (t)-\bar{u}_k (t) \right|}.
$$

The following inequality holds for the difference between the functions $\tilde{u}_k(t)$ and $\bar{u}_k(t)$:
$$
     \left| \tilde{u}_k (t)-\bar{u}_k (t) \right| \le \int_{0}^{t} \left| \tilde{q}(s) - \bar{q}(s) \right |\left|{\bar{u}_{k}}(s) \right |{{(t-s)}^{\rho -1}}{{E}_{\rho ,\rho }}\left( -(\lambda _{k}^{2}{{m}_{\sigma }} +n_q){{(t-s)}^{\rho }} \right)ds.
$$
From this inequality and Lemma \ref{lem4.5}, the following is obtained:
$$
\left| L[\tilde{q}](t)-L[\bar{q}](t) \right|\le \frac{M_{\sigma}}{\phi_0\Gamma (\rho )}{{\left\| \tilde{q}(t) - \bar{q}(t) \right\|}_{C[0,T]}}\int_{0}^{t}{\sum\limits_{k=1}^{\infty }{\lambda _{k}^{3}\left| {{{\bar{u}}}_{k}}(t) \right|}}{{(t-s)}^{\rho -1}}ds
$$
$$
\le C(T) {{\left\| \tilde{q}(t) - \bar{q}(t) \right\|}_{C[0,T]}},
$$
where
$$
C(T) =\frac{l{{M}_{\sigma }}{{T}^{\rho }}}{{\sqrt{6}{\psi }_{0}}\Gamma (\rho +1)}\left( \frac{{{T}^{\rho }}}{\Gamma (\rho +1)}{{\left\| f_{xxxx}^{(4)}(x,t) \right\|}_{C([0,T];{{L}_{2}}(0,l))}}+{{\left\| \phi^{(4)}(x) \right\|}_{{{L}_{2}}(0,l)}} \right).
$$

According to condition (4) of Theorem \ref{thm4.3}, it follows that $C(T) < 1$.

\hfill $\square$

\section{Acknowledgments} The authors acknowledge financial support from the Ministry of Higher Education, Science and Innovation of the Republic of Uzbekistan, Grant No. F-FA-2021-424.


\label{lastpage}

\end{document}